\documentclass[letterpaper]{amsart}
\usepackage{amsfonts,amssymb,amscd,amsmath,enumerate,url,verbatim}
\usepackage[all]{xy}
\SelectTips{eu}{}
\xyoption{curve}
\CompileMatrices

\newcommand{\ov}{\overline}

\newcommand{\ges}{{\scriptscriptstyle\geqslant}}
\newcommand{\col}{\colon}

\newcommand{\fm}{{\mathfrak m}}
\newcommand{\fa}{{\mathfrak a}}
\newcommand{\fb}{{\mathfrak b}}
\newcommand{\fn}{{\mathfrak n}}

\newcommand{\bd}{\boldsymbol}

\newcommand{\edim}{\operatorname{edim}}
\newcommand{\reg}{\operatorname{reg}}
\newcommand{\codim}{\operatorname{codim}}
\newcommand{\depth}{\operatorname{depth}}
\newcommand{\lin}{\operatorname{lin}}
\newcommand{\ld}{\operatorname{ld}}

\newcommand{\rank}{\operatorname{rank}}
\newcommand{\HH}{\operatorname{H}}

\newcommand{\Tor}{\operatorname{Tor}}
\newcommand{\Ext}{\operatorname{Ext}}
\newcommand{\Hom}{\operatorname{Hom}}
\newcommand{\Po}{\operatorname{P}}

\newcommand{\hilb}[2]{\operatorname{Hilb}_{#1}(#2)}
\newcommand{\agr}[2][{}]{{{#2}^{\mathsf g}_{#1}}}

\newcommand{\ds}[1]{\displaystyle{#1}}

\everymath{\displaystyle}

\theoremstyle{plain}
\newtheorem{theorem}{Theorem}[section]

\newtheorem{proposition}[theorem]{Proposition}
\newtheorem{lemma}[theorem]{Lemma}
\newtheorem{corollary}[theorem]{Corollary}

{\Alph{theorem}}
{\Alph{theorem}}

\theoremstyle{definition}

\newtheorem{chunk}[theorem]{}

\theoremstyle{remark}

\newtheorem{question}[theorem]{Question}
\newtheorem*{Question1}{Question 1}
\newtheorem*{Question2}{Question 2}
\newtheorem*{Question3}{Question 3}

\newenvironment{bfchunk}{\begin{chunk}\textit}{\end{chunk}}
\newtheorem*{Claim1}{Claim 1}
\newtheorem*{Claim2}{Claim 2}

\newtheorem{remark}[theorem]{Remark}

\numberwithin{equation}{theorem}

\begin{document}
\title[Linearity defect]{On the linearity defect of the residue field}

\begin{abstract}
Given a  commutative Noetherian local ring $R$, the linearity defect of a finitely generated $R$-module $M$, denoted $\ld_R(M)$, is an invariant that measures how far $M$ and its syzygies are from having a linear resolution. Motivated by a positive known answer in the graded case, we study the question of whether $\ld_R(k)<\infty$  implies $\ld_R(k)=0$.   We give answers in special cases, and we discuss several interpretations and refinements of the question. 
\end{abstract}

\author[L.~M.~\c{S}ega]{Liana M.~\c{S}ega}
\address{Liana M.~\c{S}ega\\ Department of Mathematics and Statistics\\
   University of Missouri\\ \linebreak Kansas City\\ MO 64110\\ U.S.A.}
     \email{segal@umkc.edu}

\subjclass[2000]{Primary 13D07. Secondary 13D02}
\keywords{linearity defect, Koszul local rings}
\thanks{Research partly supported by NSF grant DMS-1101131 and Simons Foundation grant 20903}

\maketitle

\section*{Introduction}
This paper is concerned with properties of the minimal free resolution of the residue field $k$ of a commutative Noetherian local ring $R$, and more precisely with the observation that certain behavior of the tail of the resolution determines properties of its beginning. An instance of such behavior is displayed by  Castelnuovo-Mumford regularity:  If $A$ is a standard graded algebra over a field $k$  and  $\reg_A(k)<\infty$, then $A$ is Koszul, cf.\! Avramov and Peeva \cite{AP}. 

We study an invariant related to regularity, defined for all local rings $R$: the  {\it linearity defect} of a finitely generated  $R$-module $M$, denoted $\ld_R(M)$. This notion was introduced by  Herzog and Iyengar \cite{HI} and studied further by Iyengar and R\"omer \cite{IR}.  The definition is recalled in Section 1. We denote $\agr{(-)}$ the associated graded objects with respect to the  maximal ideal $\fm$ of $R$, and we recall here that $\ld_R(M)<\infty$ if and only if there exists a syzygy $N$ of $M$ such that $\agr N$ has a linear resolution over $\agr R$, and $\ld_R(k)=0$ if and only if the algebra $\agr R$ is Koszul. 

Linearity defect can also be defined for standard graded $k$-algebras. If $A$ is a standard graded $k$-algebra, then  $\ld_A(k)<\infty$ implies $\reg_A(k)=0$, cf.\! \cite[1.12, 1.13]{HI}, hence $A$ is Koszul. A natural question, raised in \cite{HI}, is whether this property can be extended to local rings $(R,\fm,k)$:

\begin{Question1}
If $\ld_R(k)<\infty$, does it follow that $\ld_R(k)=0$? 
\end{Question1}

We give an interpretation of the linearity defect in terms of the  maps
$$
\nu_i^n(M)\colon \Tor_i^R(M, R/\fm^{n+1})\to \Tor_i^R(M,R/\fm^{n})\,
$$
induced by the canonical surjection $R/\fm^{n+1}\to R/\fm^n$. More precisely, we show in Theorem \ref{mu} that  $\ld_R(M)<\infty$ if and only if $\nu^n_{\gg 0}(M)=0$ for all $n>0$, and $\ld_R(M)=0$ iff $\nu^n_{> 0}(M)=0$ for all $n>0$. (The notation $\nu_{\gg 0}^n(M)=0$ means $\nu^n_i(M)=0$ for all $i\gg 0$ and the notation  $\nu^n_{> 0}(M)=0$ means  $\nu_i^n(M)=0$ for all $i>0$.) Based on this interpretation, Question 1 can be refined as follows: 

\begin{Question2}
Let $n>0$. If $\nu_{\gg 0}^n(k)=0$, does it follow that $\nu_{>0}^n(k)=0$? 
\end{Question2}

The maps $\nu^1(k)$  are connected to the Yoneda algebra $E=\Ext_R(k,k)$. Let $R^{!}$ denote the subalgebra of $E$ generated by $\Ext^1_R(k,k)$. A result of Roos shows that $\nu_{> 0}^1(k)=0$ if and only if  $E=R^{!}$ and $\nu_{\gg 0}^1(k)=0$ if and only if $E$ is finitely generated as a module over $R^{!}$.  When $n=1$, Question 2 can be reformulated thus as follows: 

\begin{Question3}
If $E$ is finitely generated as a module over $R^{!}$, does it follow that $E=R^{!}$? 
\end{Question3}

As mentioned above, Question 1 has a a positive answer for standard graded $k$-algebras. In this case, we show that Question 3 has a  positive answer, as well; see Corollary \ref{gradedExt}. 

We further provide answers to these questions for a local ring $(R,\fm,k)$, as indicated below: 

\begin{enumerate}
\item Question 1 has a positive answer under any of the following assumptions: 
\begin{enumerate}[\quad \rm(a)]
\item $\fm^3=0$ (Section 1);
\item $R$ is complete intersection and $\agr R$ is Cohen-Macaulay (Section 6);
\item $R$ is Golod and $\agr R$ is Cohen-Macaulay (Section 7).
\end{enumerate}  

\item Question 2 has a positive answer under any of the following assumptions:
\begin{enumerate}[\quad\rm(a)]
\item $R$ is complete intersection and $n=1$ (Section 5);
\item $\fm^{n+2}=0$.  (Section 7).
\end{enumerate}

\item Question 3 has a positive answer under any of the following assumptions:
\begin{enumerate}[\quad\rm(a)]
\item $R$ is complete intersection (Section 5);
\item $E$ is generated over $R^!$ in degree $2$ (Section 5).
\end{enumerate}
\end{enumerate}

The paper is organized as follows. 

 In Section 1 we mainly introduce terminology.  In Section 2 we prove the interpretation of linearity defect $\ld_R(M)$ in terms of the maps $\nu^n_i(M)$. 

In Section 3 we note that $\ld_R(k)=0$ implies $\ld_R(\fm^n)=0$ for all $n$  and we provide a (possible) generalization:  We show that  $\ld_R(\fm^n)\le\ld_R(k)$ for all $n>0$. We ask whether equality holds; note that a positive answer for $n=1$ would give a positive answer to Question 1. 

In Section 4 we use a minimal Tate resolution of $k$ to prove that $\nu_{\gg 0}^1(k)=0$ implies $\nu_2^1(k)=0$.  The latter condition is shown in Section 5 to be equivalent to the following {\it quadratic}  property : If $R\cong \widehat Q/\fa$ is a minimal Cohen presentation with $(Q,\fn,k)$ regular local and $\fa\subseteq \fn^2$, then $\fn^3\cap \fa\subseteq \fn \fa$. 

Section 5 makes the connections with the Yoneda algebra. Section 6 is concerned with complete intersection rings, and Section 7 with Artinian and Golod rings. 

\section{Linearity defect}
In this section we establish notation, provide the definition of linearity defect, and we introduce Question 1.  Proposition \ref{m^3} provides an easily available answer in the radical cube zero case. 

Throughout, $(R,\fm,k)$ denotes a commutative Noetherian local ring; the notation identifies $\fm$ as the maximal ideal and $k$ as the residue field. In addition, we  assume that $\fm\ne 0$.  The {\it embedding dimension} of $R$, denoted $\edim R$, is the minimal number of generators of $\fm$, and the {\it codimension of $R$}, denoted $\codim R$, is defined as the number $\edim R-\dim R$.

If $M$ is a finitely generated $R$-module, then  the $i$th {\it Betti number} of $M$ is the number
\[
\beta_i^R(M)=\rank_k\Tor_i^R(M,k)\,.
  \] 
The \emph{Poincar\'e series} of $M$ over $R$ is the formal
power series
\[
  \label{Poincare}
  \Po_M^R(t)=\sum_{i\ge 0}\beta_i^R(M)t^i \in\mathbb Z[\![t]\!]\,.
\]
The {\it Hilbert series}  of $M$ is the formal power series 
$$
\hilb Mt=\sum_{n\ge 0}\rank_k(\fm^nM/\fm^{n+1}M)t^n\in\mathbb Z[\![t]\!]\,.
$$

We let $\ds{\agr R}$ denote the associated graded ring and  $\ds{\agr M}$ denote the associated graded module with respect to $\fm$; that is,
\[\ds{\agr R=\bigoplus_{n\geq 0}\ \fm ^n /\fm^{n+1}}
\quad\text{and}\quad
\ds{\agr M=\bigoplus_{n\geq 0}\ \fm ^n M/\fm^{n+1} M}.\]

\begin{bfchunk}{Linearity defect.}
\label{ld}
Let $M$ be a finitely generated $R$-module and let 
$$
F=\qquad  \cdots\to F_{i+1}\xrightarrow{\partial_i} F_i\to\cdots\to F_0\to 0
$$
be a minimal free resolution of $M$, with differential $\partial$.

For any such $F$ one constructs the complex 
\[
\operatorname{lin}^{R}(F)=\ \cdots\to\agr{F_{i+1}}(-i-1)\to \agr{F_{i}}(-i)\to \cdots\to \agr{F_{0}}\to 0\]
with differentials induced from $F$. Herzog and Iyengar \cite[1.7]{HI} defined the {\it linearity defect} of $M$ to be the number:
$$
\ld_R(M)=\sup\{i\in\mathbb Z\mid \HH_i(\lin^R(F))\ne 0\}\,.
$$
We make the convention that $\ld_R(0)=0$. 

If $\ld_R(M)=0$, we say, following \cite{HI}, that $M$ is a {\it Koszul module.} If $k$ is Koszul, we say that $R$ is a {\it Koszul ring}.  The connection with the classical Koszul algebra notion, defined for standard graded rings,  is as follows:  $R$ is a Koszul ring iff $\agr R$ is a Koszul algebra. Also, $M$ is Koszul iff $\agr M$ has a linear resolution over $\agr R$. 

For an integer $d$, one has $\ld_R(M)\le d$ if and only if the $d$th syzygy module in a minimal free resolution of $M$ over $R$ is Koszul. 

If $M$ is a Koszul $R$-module, then $\lin^R(F)$ is a minimal free resolution of $\agr M$ over $\agr R$. Since this resolution is minimal and  Hilbert series are additive on short exact sequences, we have:  
\begin{equation}
\label{KoszulP}
\Po_M^R(t)=\Po_{\agr M}^{\agr R}(t)=\frac{\hilb M{-t}}{\hilb R{-t}}\,.
\end{equation}
\end{bfchunk}

\begin{remark}
\label{graded1}
Let $A$ be a standard graded algebra over a field $k$, and let $\fm$ denote its maximal irrelevant ideal. When $N$ is a finitely generated graded $A$-module, one can define in the same manner the linearity defect $\ld_A(N)$ of $N$,  by using a minimal graded free resolution of $N$ over $A$. All subsequent definitions and results can be similarly adapted to the graded case. 
\end{remark}

As mentioned in the introduction, we are concerned with the following question: 

\begin{Question1}
\label{question}
If $\ld_R(k)<\infty$ does it follow that $R$ is Koszul?
\end{Question1}

This  question was posed in \cite[1.14]{HI},  motivated by the fact that a positive answer is available in the case of standard graded algebras; see \cite[1.13]{HI} for a proof of this fact,  and  \cite[2.4]{HI} for a further discussion of the problem. 

A positive answer in the case $\fm^3=0$ is easily available.

\begin{proposition}
\label{m^3}
Assume  $\fm^3=0$. If there exists a finitely generated $R$-module $M$ such that  $\ld_R(M)<\infty$, then $R$ is Koszul. 
\end{proposition}

\begin{proof}
Since $\ld_R(M)<\infty$, there exists a syzygy $N$ in a minimal free resolution of $M$ such that $\ld_R(N)=0$. Since $\fm^3=0$, one has that $\fm^2N=0$. The $\agr R$-module $\agr N$ then has a linear resolution, hence $\agr R$ is Koszul, according to Avramov {\it et.\,al.} \cite[Theorem 1.6]{AIS}
\end{proof}

\section{The maps $\nu^n$}

In this section we introduce the maps $\nu^n$ and begin their study. The main result is Theorem \ref{mu}, which provides an interpretation of linearity defect in terms of these maps, and motivates Question 2.   A connection with Castelnouvo-Mumford regularity is made in Proposition \ref{gradedmu1}, based on the proof of results of Herzog and Iyengar \cite{HI}.

\begin{bfchunk}
Let $M$ be a finitely generated $R$-module and $F$ a minimal free resolution of $M$. For all  integers  $n$ and $i$ we consider the  map 
\begin{equation}
\label{mudef}
\nu^n_i(M)\colon \Tor^R_i(M, R/\fm^{n+1})\to \Tor^R_i(M,R/\fm^{n})
\end{equation}
 induced in homology by the canonical surjection
$R/\fm^{n+1}\to R/ \fm^{n}$. 

When $M=k$, we drop the module argument: We set  $\nu^n=\nu^n(k)$. 
\end{bfchunk}

\begin{theorem}
\label{mu}
Let $M$ be a finitely generated  $R$-module. The following then hold: 
\begin{enumerate}[\quad\rm(a)]

\item  If $i>0$ and $F$ is a minimal free resolution of $M$, then  $\HH_i(\operatorname{lin}^{R}(F))=0$ if and only if  $\nu^n_{i+1}(M)=0=\nu^n_{i}(M)$ for all $n>0$. 
\item $\ld_R(M)\le d$ if and only if $\nu^n_i(M)=0$ for all $i> d$ and all $n\ge 0$. 
\item $\ld_R(M)=0$ if and only if $\nu^n_i(M)=0$ for all $i>0$ and all $n\ge 0$. 
\end{enumerate}
\end{theorem}

Obviously, parts (b) and (c) are immediate consequences of part (a). The proof of (a) will be given after some preliminaries. 

If $s$ is an integer such that $\nu^n_i=0$ for all $i>s$, we write $\nu^n_{>s}=0$. The notation $\nu^n_{\gg 0}=0$ means that an integer $s$ as above exists.  Note that the  statement of Theorem \ref{mu} motivates  Question 2 from the introduction.

We proceed now with some preparation for the proof of Theorem \ref{mu}(a). 

\begin{chunk}
\label{interpret}
Let $i>0$ and $n\ge 0$. Computing homology in \eqref{mudef} by means of a minimal free resolution $F$ of $M$ and using the natural isomorphism  $F\otimes_RR/\fm^k\cong F/\fm^kF$ for $k=n,n+1$, the map $\nu^n_i(M)$ can be realized as the map 
$$H_i(F/\fm^{n+1}F)\to H_i(F/\fm^nF)$$
induced in homology by the canonical surjection of complexes $F/\fm^{n+1}F\to F/\fm^nF$. 

The following statements are thus equivalent:
\begin{enumerate}
\item $\nu_i^n(M)=0$;
\item For every $x\in F_{i}$ with $\partial x\in \fm^{n+1}F_{i-1}$ there exists $y\in F_{i+1}$ such that $x-\partial y\in \fm^nF_i$;
\item For every $x\in F_{i}$ with $\partial x\in \fm^{n+1}F_{i-1}$ there exists $u\in \fm^nF_i$ such that $\partial x=\partial u$.
\end{enumerate}
When $n=1$, these conditions are also equivalent with 
\begin{enumerate}
\item[(2\rq{})] If $x\in F_{i}$ satisfies  $\partial x\in \fm^2F_{i-1}$, then $x\in \fm F_i$.
\end{enumerate}
For the proof of the implication (2)$\implies$(3) take $u=x-\partial y$. For the  proof of the implication (3)$\implies$(2), take $y$ such that $\partial y=x-u$, using the fact that $H_i(F)=0$. 
\end{chunk}

\begin{remark}
\label{mu_1}
Note that $\nu_1^n=0$ for all $n\ge 0$. Indeed, let $g_1, \dots, g_e$ be a minimal generating set for $\fm$ and consider a minimal free resolution $F$ of $k$ with $F_0=R$ and $F_1=R^e$. Let $x\in F_1$ with $\partial x\in \fm^{n+1}$. Then there exist $u_1, \dots, u_e\in\fm^n$ such that $\partial x=u_1g_1+\dots+u_eg_e$. Let $f_1, \dots, f_e\in F_1$ such that $\partial f_i=g_i$ for all $i$. If $u=u_1f_1+\dots+u_ef_e$, then we have $\partial x=\partial u$  and $u\in \fm^nF_1$, hence condition (3) of \ref{interpret} is satisfied. 
\end{remark}

\begin{remark}
\label{flat}
Let $\varphi\colon (R,\fm,k)\to (R',\fm',k)$ be a flat homomorphism of local rings such that $\varphi(\fm)=\fm'$; in particular, $\varphi$ is faithfully flat.  Let $M$ be a finitely generated $R$-module and set $M'=M\otimes_RR'$. When considering the maps $\nu^n(M')$, we regard $M'$ as an $R'$-module. Standard arguments show  that for all integers $n$, $i$ we have  $\nu^n_i(M)=0$ iff $\nu^n_i(M')=0$. Also, if $F$ is a minimal free resolution of $M$ over $R$, then $F'=F\otimes_RR'$ is a minimal free resolution of $F'$ over $R'$ and $\HH_i(\lin_R(F))=0$ if and only if $\HH_i(\lin_{R'}(F'))=0$. In particular, $\ld_R(M)=\ld_{R'}(M')$. 
\end{remark}

\begin{bfchunk}{Notation.}
\label{notation} Set $U=\operatorname{lin}^{R}(F)$. Note that $U$ is a complex of graded $\agr R$-modules, and we will denote by $\HH_i(U)_{j}$ the $j$th graded component of the $i$th homology module $\HH_i(U)$. For all integers $j$ and $s$, one has that $\HH_j(U)_{j+s}$ is the homology of the complex: 
\begin{equation}
\label{U}
\fm^{s-1}F_{j+1}/\fm^sF_{j+1}\to\fm^sF_j/\fm^{s+1}F_j\to\fm^{s+1}F_{j-1}/\fm^{s+2}F_{j-1}
\end{equation}
with differentials induced by the differential $\partial$ of the complex $F$. We will use $\partial$ to denote the differential of this complex, as well.  If $x\in\fm^{s} F_j$ we  denote $\ov x$ the image of $x$ in $\fm^sF_j/\fm^{s+1}F_j$. Note that $\partial x\in \fm^{s+1}F_{j-1}$ and, with the notational convention  above, one has $\partial({\ov x})=\ov{\partial x}$. 
\end{bfchunk}

\begin{proof}[Proof of Theorem {\rm \ref{mu}(a)}]

In Remark \ref{flat}, we may take  $R'$ to be the $\fm$-adic completion of $R$. Consequently, we may assume that $R$ is complete. 
Let $U$ be as in \ref{notation}. 

Let $i>0$. Assume that $\HH_i(U)=0$, hence $\HH_i(U)_{i+s}=0$ for all $s$. Let $n\ge 0$. We want to prove that  $\nu^n_{i+1}(M)=0=\nu^n_{i}(M)$. To do so, we will verify condition (3) of \ref{interpret}, with the appropriate indices.

To show $\nu_{i+1}^n(M)=0$, let $x\in F_{i+1}$ with $\partial x\in \fm^{n+1}F_i$. Note that  $\ov{\partial x}$ is a cycle in $U_{i,i+n+1}$. Since $H_i(U)_{i+n+1}=0$, we have that $\ov{\partial x}=\ov{\partial u_0}$ for some $u_0\in \fm^nF_{i+1}$. Then $\partial(x-u_0)\in \fm^{n+2}F_i$. Using inductively this reasoning, we obtain that for each $k\ge 0$ there exists $u_k\in \fm^{n+k}F_i$ such that $\partial x-\partial u_0-\partial u_1-\cdots -\partial u_k\in\fm^{n+2+k}F_i$. Since $R$ is assumed complete, we can set $u=u_0+u_1+\cdots$. Then $u\in \fm^nF_{i+1}$ and $\partial x-\partial u=0$. This shows that  $\nu_{i+1}^n(M)=0$, using \ref{interpret}.

We now show $\nu^n_{i}(M)=0$. We will prove this by induction on $n\ge 0$. Obviously, one has $\nu^0_{i}(M)=0$. 
Assume now that $n>0$ and $\nu^k_{i}(M)=0$ for $k<n$. Let $x\in F_{i}$ with $\partial x\in \fm^{n+1}F_{i-1}$.  In particular, $\partial x\in \fm^nF_{i-1}$. Since we assumed $\nu^{n-1}_i(M)=0$, condition (3) of \ref{interpret} gives that  there exists  $u'\in \fm^{n-1}F_i$ such $\partial x=\partial u'$.
 Considering $\ov{u'}$ as an element in $U_{i,i+n-1}=\fm^{n-1} F_i/\fm^n F_i$, we have $\partial{\ov { u'}}=0$. Since $\HH_i(U)_{i+n-1}=0$, there exists $\ov z\in \fm^{n-2}F_{i+1}/\fm^{n-1}F_{i+1}$ such that  $\ov{u'}=\ov{\partial z}$ in $ \fm^{n-1}F_{i}/\fm^{n}F_{i}$,  hence $u'-\partial z\in \fm^n F_i$. We take then $u=u'-\partial z$ and note that $u\in \fm^{n}F_i$ and $\partial x=\partial u$.  This shows $\nu^n_i(M)=0$, using \ref{interpret}.  
\smallskip

Assume now that $\nu^n_{i+1}(M)=\nu^n_{i}(M)=0$ for all $n\ge 0$; we use again \ref{interpret} to translate below these conditions.  Let $s\ge 0$. We will show that  $\HH_i(U)_{i+s}=0$.  Let $\ov x\in \fm^sF_i/\fm^{s+1}F_i$ with $\partial\ov x=0$ in $\fm^{s+1}F_{i-1}/\fm^{s+2}F_{i-1}$. Thus $\partial x\in\fm^{s+2}F_{i-1}$. Since $\nu^{s+1}_{i}(M)=0$, there exists $y\in \fm^{s+1}F_{i}$ and $u\in F_{i+1}$ such that  $x-y=\partial u$.  Then  $\partial u\in \fm^sF_i$ and  the fact that  $\nu^{s-1}_{i+1}(M)=0$ shows that $\partial u=\partial w$ with $w\in \fm^{s-1}F_{i+1}$. Thus $x=y+\partial w$, with $y\in \fm^{s+1}F_i$ and  $w\in \fm^{s-1}F_{i+1}$, hence $\ov x=\partial\ov w$ in $U$, showing that $\HH_i(U)_{i+s}=0$.
\end{proof}

In the graded case, the maps $\nu^1$ are related to the  Castelnuovo-Mumford regularity, denoted $\reg(-)$. Analyzing the proof \cite[1.12, 1.13]{HI}, we identify below a weaker hypothesis, in terms of the maps $\nu^1$.  We repeat the main steps of the proofs given  in {\it  loc.\,cit.}, with the purpose of making sure that the weaker hypothesis is indeed sufficient.

\begin{proposition}
\label{gradedmu1}
Let $A$ be a standard graded $k$-algebra, and $N$ a graded finitely generated $A$-module. If $\nu^1_{\gg 0}(N)=0$, then $\reg_A(N)<\infty$. 

In particular, if $\nu^1_{\gg 0}(k)=0$, then $A$ is Koszul. 
\end{proposition}

\begin{proof} (Following \cite[1.12, 1.13]{HI}.)
We denote $\fm$  the maximal irrelevant ideal of $A$. Let $F$ be a minimal free graded resolution of $N$ over $A$. As described in \cite[1.11]{HI}, $U=\lin^R(F)$ decomposes as a direct sum of the linear strands complexes $F^r$, defined in {\it loc.\,cit.}, and $\reg_A(N)$ is equal to $\sup\{r\in\mathbb Z\mid F^r\ne 0\}$.

For each $r$, let $n(r)$ denote the least integer such that $(F^r)_{n(r)}\ne 0$. If $\reg_A(N)$ is infinite, then, as noted in the proof of \cite[1.12]{HI}, there exists an infinite sequence $r_1<r_2<\cdots$ with $n(r_1)<n(r_2)<\cdots$. Since $(F^{r_s})_n=0$ for $n<n(r_s)$ and $F^{r_s}$ is minimal, it follows that there exists a nonzero element of bidegree $(n(r_s), n(r_s)+r_s)$ in $F^{r_s}\otimes k$.  This element corresponds to a nonzero element of bidegree $(n(r_s),n(r_s))$ in $\HH(U)$. In particular, there exists an element $x_s$  in $F_{n(r_s)}\smallsetminus \fm F_{n(r_s)}$  such  that $\partial x_s\in \fm^2F_{n(r_s)-1}$, and this shows, using \ref{interpret}(2'), that $\nu^{1}_{n(r_s)}(N)\ne 0$ for each $s$.  This contradicts the assumption that $\nu^1_{\gg 0}(N)=0$.

For the last statement, one uses the result of Avramov and Peeva \cite[(2)]{AP}: If $\reg_A(k)<\infty$, then $A$ is Koszul. 
\end{proof}

Applying Theorem \ref{mu}, one obtains the statement of \cite[1.13]{HI}: 
\begin{corollary}
Let $A$ be a standard graded algebra. If $\ld_A(k)<\infty$, then $A$ is Koszul. 
\end{corollary}

\section{The linearity defect of powers of the maximal ideal}

The powers of the maximal ideal and, more generally, modules of the form $\fm^nM$ with $n>0$ and $\fm^nM\ne 0$, are known to share some of the asymptotic properties of the residue field; for example, they have the same complexity. With this thought in mind, we proceed to investigate the connection between $\ld_R(k)$ and $\ld_R(\fm^n)$. We prove in Proposition \ref{ineq} an inequality between these numbers.  

\begin{remark}
\label{ld0}
If $\ld_R(k)=0$, then $\ld_R(\fm^n)=0$ for all $n$. 

Indeed, if $\ld_R(k)=0$ then $R$ is Koszul  and $\agr R$ is a Koszul algebra. Then  $(\agr \fm)_{\ge n }$ has a linear resolution over $\agr R$ by  \cite[5.4]{IR}, hence $\ld_R(\fm^n)=0$.
\end{remark}

The next proposition provides a possibly more general result. 

\begin{proposition}
\label{ineq}
$\ld_R(\fm^n)\le \ld_R(k)$ for all $n$. 
\end{proposition}

\begin{proof}
 Given a finite $R$-module $M$, recall that
$$
\nu^n_i(M)\colon \Tor_i^R(M, R/\fm^{n+1})\to \Tor_i^R(M,R/\fm^{n})
$$
is induced by the surjection $R/\fm^{n+1}\to R/\fm^n$. We (ab)use the same notation for the avatar of this map which arises by switching the module entries: 
$$
\nu^n_i(M)\colon \Tor_i^R(R/\fm^{n+1}, M)\to \Tor_i^R(R/\fm^{n},M)\,.
$$
We set $\ov M=M/\fm M$.

Let $i>1$ and $p,q\ge 0$.  Consider the exact sequence
$$
0\to\ov{\fm^n}\to R/\fm^{n+1}\to R/\fm^n\to 0
$$
for $n=p$ and $n=q$ and the induced commutative  diagram with exact rows and columns.
\[
\xymatrixrowsep{1.9pc}
\xymatrixcolsep{4.1pc}
\xymatrix{
\Tor_{i}^R(\ov{\fm^q},\frac{R}{\fm^{p+1}})\ar@{->}[d]^{\gamma}\ar@{->}[r]^{\quad \nu_{i}^{p}(\ov{\fm^{q}})}&\Tor_{i}^R(\ov{\fm^q},\frac{R}{\fm^p})\ar@{->}[d]^{\partial}\ar@{->}[r]&\Tor_{i-1}^R(\ov{\fm^q},\ov{\fm^p})\ar@{->}[d]\\
\Tor_{i}^R(\frac{R}{\fm^{q+1}},\frac{R}{\fm^{p+1}})\ar@{->}[r]^{\quad \nu_{i}^{p}(R/\fm^{q+1})\quad}\ar@{->}[d]^{\nu_{i}^{q}(R/\fm^{p+1})}&\Tor_{i}^R(\frac{R}{\fm^{q+1}},\frac{R}{\fm^p})\ar@{->}[r]^{\alpha}\ar@{->}[d]^{\nu_{i}^{q}(R/\fm^{p})}&\Tor_{i-1}^R(\frac{R}{\fm^{q+1}},\ov{\fm^p})\ar@{->}[d]^{\nu_{i-1}^{q}(\ov{\fm^{p}})}\\
\Tor_{i}^R(\frac{R}{\fm^q},\frac{R}{\fm^{p+1}})\ar@{->}[r]^{\quad \nu_{i}^{p}(R/\fm^{q})}&\Tor_{i}^R(\frac{R}{\fm^q},\frac{R}{\fm^p})\ar@{->}[r]^{\beta}&\Tor_{i-1}^R(\frac{R}{\fm^q},\ov{\fm^p})\\
}
\]

Set $\ld_R(k)=d$ and assume $i>d+1$.  Then $\nu_{i-1}^{n}=\nu_{i}^n=0$ for all $n\ge 0$,  by Theorem \ref{mu}.  Since $\ov{\fm^p}$ and $\ov{\fm^q}$ are $k$-vector spaces, we also have $\nu_{i-1}^{q}(\ov{\fm^{p}})=\nu_{i}^{p}(\ov{\fm^{q}})=0$. 

\begin{Claim1}
 If $\nu_i^{q}(R/\fm^{p+1})=0$, then $\nu_i^{p}(R/\fm^{q+1})=0$.
\end{Claim1}

\begin{Claim2}
If $\nu_i^{p}(R/\fm^q)=0$, then $\nu_i^{q}(R/\fm^{p})=0$.
\end{Claim2}

Claim 1 is obtained  by analyzing the upper left square  of the diagram: If $\nu_i^{q}(R/\fm^{p+1})=0$, then $\gamma$ is surjective. Since $\nu_{i}^{p}(\ov{\fm^{q}})=0$, the commutativity of the square gives that $\nu_i^{p}(R/\fm^{q+1})\circ \gamma=0$, hence $\nu_i^{p}(R/\fm^{q+1})=0$. 

Claim 2 is obtained  by analyzing the lower right square of the diagram: If $\nu_i^{p}(R/\fm^{q})=0$, then the  map $\beta$ is injective. Since $\nu_{i-1}^{q}(\ov{\fm^p})=0$, the commutativity of the square  yields that $\beta\circ \nu_i^{q}(R/\fm^p)=0$, hence $\nu_i^{q}(R/\fm^p)=0$.

Hence these claims hold for all $p,q\ge 0$. Note that $\nu_i^p(R/\fm^{q+1})=0$ for $q=0$, since $\nu_i^p=0$. 
We then apply  Claim 2  and we get $\nu_i^1(R/\fm^{p})=0$. Then we apply Claim 1 and we get $\nu_i^{p-1}(R/\fm^2)=0$. Then Claim 2 yields $\nu_i^2(R/\fm^{p-1})=0$. An inductive repeated use of Claim 1 and Claim 2 then yields that $\nu_i^p(R/\fm^q)=0$ for all $p, q\ge 0$ and all $i>d+1$. Using Theorem \ref{mu}  we conclude $\ld_R(R/\fm^p)\le  d+1$ for all $p\ge 0$.

 If $n>0$ and $\fm^n\ne 0$,  then $\fm^n$ is a first syzygy in a minimal free resolution of $R/\fm^n$ and we conclude 
$\ld_R(\fm^n)\le \max\{0,\ld_R(R/\fm^n)-1\}\le d=\ld_R(k)$.
\end{proof}

\begin{question}
Is it true that $\ld_R(\fm^n)=\ld_R(k)$ for all $n>0$? 
\end{question}

Note that for $n=1$ this question subsumes Question 1 in the introduction, for if $\ld_R(k)<\infty$ then $\ld_R(\fm)=\ld_R(k)-1$, unless $\ld_R(k)=0$. 

\section{Tate resolutions and the map $\nu^1$}

 In this section we are concerned with Question 2 in the introduction, for $n=1$: If $\nu^1_{\gg 0}=0$, does it follow that $\nu^1_{>0}=0$?  Proposition \ref{13}, whose proof uses Tate resolutions,  gives a partial answer. 

Recall  that $\nu^n_1=0$ for all $n\ge 0$, as noted in Remark \ref{mu_1}.  In particular, the condition $\nu_1^1=0$ holds for all local rings $(R,\fm,k)$.  However, as we will see in Section 5, the condition $\nu^1_2=0$ does not hold for all rings.  

\begin{proposition}
\label{13}
If $\nu^1_{2n}=0$ for some $n> 0$, then $\nu^1_2=0$. 

\end{proposition}

Some preliminaries are needed.

\begin{bfchunk}{Divided powers and Tate resolutions.}
A system of divided powers on a graded $R$-algebra $A$ is an operation that for
each $j\ge 1$ and each $i\ge 0$ assigns to every $a\in A_{2i}$ an element $a^{(j)}\in A_{2ij}$ , subject to certain axioms; cf. \cite[1.7.1]{GL}. A DG$\Gamma$ R-algebra is a DG $R$ -algebra $A$ with divided powers compatible with the differential $\partial$ of $A$: $\partial(a^{(j)}) = \partial(a) a^{(j-1)}$. We denote by $|x|$ the homological degree of an element $x$. 

Given a set $\bd x=\{x_i\mid |x_i|\ge 1\}$, we let $A\langle{\bd x}\rangle$
denote a DG$\Gamma$ algebra with
  \begin{gather*}
A\otimes_R\mathsf\Lambda_*^R\bigg(\bigoplus_{\substack{x\in\bd x\\
 |x|\text{ odd}}}Rx\bigg)\otimes_R{\mathsf\Gamma}^R_*\bigg(\bigoplus_{\substack{x\in\bd x\\
 |x|\text{ even}}}Rx\bigg)
  \end{gather*}
as underlying graded algebra and differential compatible with that of $A$
and the divided powers of $x\in\bd x$. 

A Tate resolution of a surjective ring homomorphism $R\to T$ is a 
quasi-isomor\-phism $R\langle\bd x\rangle\to T$, where $\bd x=\{x_i\}_{i\ges 1}$
and $|x_j|\ge|x_i| \ge1$ holds for all $j\ge i\ge1$.  Such a resolution always exists: 
see  \cite[1.2.4]{GL}. Furthermore, such a resolution can be chosen minimally, as described in \cite[Construction 6.3.1]{Avr98}; for $T=k$ this construction yields a minimal free resolution of $k$, which we shall call a {\it minimal Tate resolution} of $k$ over $R$. 

For each sequence of integers $\mu<\dots<\nu$ and each sequence of integers $i_{\mu}\ge 1, \dots, i_{\nu}\ge 1$, the product $x_\mu^{(i_{\mu})}\dots x_{\nu}^{(i_{\nu})}$ is called a {\it normal $\Gamma$-monomial}; $1$ is considered to be a normal monomial. The normal monomials form the {\it standard basis} of $R\langle\bd x\rangle$, considered as a graded algebra over $R$. 

For each $\mu>0$ we set
$$
I_{\mu}=\{i\ge 0\colon |x_i|=\mu\}\,.
$$
\end{bfchunk}

\begin{proof}[Proof of Proposition {\rm \ref{13}}]
Let $F=R\langle\bd x\rangle$ be a minimal  Tate resolution of $k$. Let  $\mathcal B$  be the standard basis of  $F$ over $R$, as described above.  We will interpret the vanishing of the maps $\nu_i^1$ in terms of condition (2') in \ref{interpret}, which states that $\nu^1_i=0$ if and only if the following holds:  If $A\in F_{i}$ satisfies $\partial A\in \fm^2F_{i-1}$, then $A\in \fm F_i$. 

Assume that $\nu^1_{2n}=0$ and let $A\in F_2$ such that $\partial A\in \fm^2F_1$.  Expressing this element in terms of the basis $\mathcal B$, we have:
$$
A=\sum_{l\in I_2} a_lx_l+\sum_{i,j\in I_1, i<j} b_{ij}x_ix_j
$$
with $a_l, b_{ij}\in R$ for $l\in I_2$ and $i,j\in I_1$ with $i<j$.  Since $A^{(n)}\in F_{2n}$ satisfies $\partial A^{(n)}=\partial A\cdot A^{(n-1)}\in \fm^2F_{2n-1}$, the hypothesis that $\nu^1_{2n}=0$ yields $A^{(n)}\in \fm F_{2n}$. The coefficient of $x_l^{(n)}$ in the expression of $A^{(n)}$ in terms of the basis $\mathcal B$ is $a_l^{n}$, hence $a_l^n\in \fm$ and thus $a_l\in \fm$. Consider now the element: 
$$
A'=A-\sum_{l\in I_2} a_lx_l=\sum_{i,j\in I_1, i<j} b_{ij}x_ix_j
$$
To show $A\in \fm F$, it suffices to show $A'\in \fm F$. Note that $\partial A'\in \fm^2F$, since $\partial A\in \fm^2F$ and $a_l\in \fm$. 

If $i>j$, we set $b_{ij}=b_{ji}$. We compute next $\partial A'$, and we note that for each $x_j$ with $j\in I_1$, the coefficient of $x_j$ in $\partial A'$ is 
$$
c_j=\sum_{i\in I_1, i\ne j}\pm b_{ij}\partial x_i
$$
Since $\{x_j\}_{j\in I_1}$ is a basis for $F_1$ and $\partial A'\in \fm^2F$,  we conclude  $c_j\in \fm^2$. Recall that $F$ is a minimal resolution of $k$. In particular,  $\{\partial x_i\}_{i\in I_1}$ is a minimal generating set for  $\fm$. We conclude that $b_{ij}\in \fm$, hence $A'\in \fm F$ and thus $A\in \fm F$. 
\end{proof}

The next result will be needed later.

\begin{lemma}
\label{Tor2map}
Let  $R'=R/I$, with  $I\subseteq \fm^2$.  Let $\widehat R=Q/\fa$ be a minimal Cohen presentation of $R$, with $(Q,\fn,k)$ a regular local ring and $\fa\subseteq \fn^2$, and write $\widehat R'=Q/\fb$ for some $\fb$ with $\fa\subseteq \fb\subseteq \fn^2$. 

Then the map $\Tor_2^{\varphi}(k,k)\colon \Tor_2^R(k,k)\to\Tor_2^{R'}(k,k)$ induced by the projection $\varphi\colon R\to R'$ is injective if and only if $\fa\cap\fn\fb\subseteq \fn\fa$.
\end{lemma}

Some preliminaries are needed for the proof. 

\begin{chunk}
Let $A$ be a DG$\Gamma$ algebra and let $A_{>0}$ denote the ideal of elements of positive degree. The module of indecomposables of $A$ is the quotient of $A_{>0}$ by the submodule generated by all elements of the form $uv$ with $u,v\in A_{>0}$ and $w^{(n)}$ with $w\in A_{2i}$, for $n\ge 0$, $i>0$. We denote by $\pi_*(R)$ the module of $\Gamma$-indecomposables of $\Tor_{*}^R(k,k)$, where the DG$\Gamma$ algebra structure on $\Tor_{*}^R(k,k)$ is induced from a minimal Tate resolution of $k$. A surjective homomorphism $\varphi\colon R\to R'$ of local rings induces canonically a map 
$
\pi_*(\varphi)\colon \pi_*(R)\to\pi_*(R')$. 
\end{chunk}

\begin{proof}[Proof of Lemma {\rm \ref{Tor2map}}]
The injectivity of the map $\Tor_2^{\varphi}(k,k)$  is invariant under completion, hence we may assume $R=Q/\fa$ and $R'=Q/\fb$. 

Note that $\pi_1(\varphi)=\Tor_1^{\varphi}(k,k)$ is an isomorphism, since both $\pi_1(R)=\Tor_1^R(k,k)$ and $\pi_1(R')=\Tor_1^{R'}(k,k)$ can be canonically identified with $\fn/\fn^2$. 

In view of \cite[Corollary 1.3(b)-(c)]{Av}, the kernel of the map $\Tor_2^{\varphi}(k,k)$ can be identified with the kernel of the map $\pi_2(\varphi)$. The proof of \cite[Proposition 3.3.4]{GL} canonically identifies $\pi_2(R)$ with $\fa/\fn\fa$ and $\pi_2(R')$ with $\fb/\fn\fb$. Thus the map $\pi_2(\varphi)$ can be canonically indentified with the map $\fa/\fn\fa\to \fb/\fn\fb$ induced by the inclusion $\fa\subseteq \fb$. The kernel of this map is $(\fa\cap \fn\fb)/\fn\fa$, and the conclusion follows. 
\end{proof}

\section {Finite generation of the Yoneda algebra over the Koszul dual}

We consider now the graded algebra with Yoneda product  $E=\Ext_R(k,k)$.  Let $R^!$ denote the $k$-subalgebra of $E=\Ext_R(k,k)$ generated by its elements of degree $1$; this is sometimes referred to as {\it the Koszul dual of $R$}. Note that the Yoneda product gives $E$ a structure of right module over $R^!$. 

In this section we consider the following: 

\begin{Question3}
\label{Equestion}
If $E$ is finitely generated as a right module over $R^!$, does it follow that $E=R^!$? 
\end{Question3}

As discussed after Theorem  \ref{implications}, Question 3 arises by interpreting  Question 2 in view of the following  observation: 

\begin{remark}
\label{yoneda1}
Let $i>0$. Then $\nu^1_i=0$ if and only if the Yoneda multiplication  map 
$E^{i-1}\otimes E^1\to  E^{i}$  is surjective. 

Indeed, note that $\nu^1_i=0$ if and only if  $\Hom_R(\nu^1_i,k)=0$ if and only if the map
$
\Ext^{i}_R(k,k)\to\Ext^{i}_R(R/\fm^2,k)
$
induced by the projection $R\to R/\fm^2$ is zero. The statement then follows by applying a result of Roos \cite[Corollary 1]{Ro}. 

For $i=1$, we recover the fact that $\nu^1_1=0$.   
\end{remark}
 
Let $\rho_i\colon \Ext^i_{R/\fm^2}(k,k)\to\Ext^i_R(k,k)$ denote the graded algebra map induced by the canonical projection $R\to R/\fm^2$.   

\begin{lemma}
\label{yoneda2}
Let $i\ge  0$. Consider the following statements:
\begin{enumerate}[\quad\rm(a)]
\item $\rho_{i}$ is surjective;
\item $\rho_{i+1}$ is surjective;
\item $\nu^1_{i+1}=0$.
\end{enumerate}
Then: {\rm (b)} implies {\rm (c)}; {\rm (a)} and {\rm (c)} imply {\rm (b)}. 

In particular, $\nu^1_2=0$ if and only if $\rho_2$ is surjective. 

\end{lemma}

\begin{proof}
Consider the commutative diagram: 
\[
\xymatrixrowsep{1.7pc}
\xymatrixcolsep{3.6pc}
\xymatrix{
\Ext^{i+1}_{R/\fm^2}(k,k)\ar@{->}[r]^{\rho_{i+1}}\ar@{<-}[d]_{\cong}&\Ext^{i+1}_R(k,k)\ar@{<-}[d]\\
\Ext_{R/\fm^2}^{i}(k,k)\otimes\Ext_{R/\fm^2}^{1}(k,k)\ar@{->}[r]^{\quad\rho_i\otimes\rho_1}&\Ext_{R}^{i}(k,k)\otimes\Ext_{R}^{1}(k,k)}
\]
where the vertical maps are Yoneda products. Note that the left vertical map is an isomorphism, because $\Ext_{R/\fm^2}(k,k)$ is the free tensor algebra on $\Ext_{R/\fm^2}^1(k,k)$, cf.\ \cite[Corollary 3]{Ro}. Also, note that $\rho_1$ is an isomorphism. The conclusion then follows from the commutativity of the diagram and  \ref{yoneda1}. 
\end{proof}

We can reformulate Question 2 in the introduction, for $n=1$, in terms of the maps $\rho$:
\begin{question}
If $\rho_{\gg 0}$ is surjective, does it follow that $\rho_{>0}$ is surjective? 
\end{question}

Let $\widehat R\cong Q/\fa$ be a minimal Cohen presentation of $R$, with $(Q,\fn)$ a regular local ring and $\fa\subseteq \fn^2$.  Noting that the left-hand side is independent on the choice of the presentation, we set: 
$$
s(R)=\inf\{i\ge 1\mid \fa\cap \fn^{i+2}\subseteq \fn\fa\}
$$

\begin{lemma}
\label{s}
$\nu^1_2=0$  if and only if $s(R)=1$. 
\end{lemma}

\begin{proof}
As proved above,  $\nu^1_2=0$ if and only if $\rho_2$ is surjective. The induced map $\Hom_R(\rho_2,k)$ can be indentified with the map
$$
\Tor_2^{\varphi}(k,k)\col \Tor^R_2(k,k)\to\Tor_2^{R/\fm^2}(k,k)
$$
induced by the surjection $\varphi\colon R\to R/\fm^2$. Hence $\rho_2$ is surjective if and only if $\Tor_2^{\varphi}(k,k)$ is injective.  Apply then Lemma \ref{Tor2map}. 
\end{proof}

We are now ready to translate information about the maps $\nu^1$ in terms of  the Yoneda algebra. When we talk about finite generation of $E$ over $R^!$, we mean finite generation as a right $R^!$-module. 

\begin{theorem}
\label{implications}
Let $r\ge 1$. The following implications hold:
\[
\xymatrixrowsep{1.7pc}
\xymatrixcolsep{2.8pc}
\xymatrix{
\nu^1_{>0}=0\ar@{=>}[d]_-{\textstyle(1)\ }\ar@{<=>}[r]^-{\textstyle(5)}& E=R^{!}\ar@{=>}[d]_-{\textstyle(3)\ }\\
 \nu^1_{>r}=0\ar@{=>}[d]_-{\textstyle(2)\ }\ar@{<=>}[r]_-{\textstyle(6)\ }&\text{\begin{tabular}{c}
$E$ is generated over $R^{!}$\\
by elements of degree r
\end{tabular}} 
\ar@{=>}[d]_-{\textstyle(4)\ }\\
\nu^1_2=0\ar@{<=>}[r]^-{\textstyle(7)}&s(R)=1}
\]
\end{theorem}

\begin{proof}

The equivalences (5) and (6) are given by \ref{yoneda1}. The implication (2) is given by Proposition \ref{13}. The equivalence  (7) is given by  Lemma \ref{s}. 
\end{proof}

In view of the implications (5) and (6) in the Theorem, observe that Question 2 in the introduction can be reformulated, for $n=1$,  as follows:  Is the implication (3) reversible?  This is in fact Question 3. We give below some answers.

\begin{corollary}
If $\Ext_R(k,k)$ is generated over $R^{!}$ by its elements of degree $2$, then $\Ext_R(k,k)=R^{!}$. 
\end{corollary}

\begin{proof}
Recall that $\nu^1_1=0$. The implication (2) in Theorem \ref{implications} shows then that $\nu_{> 2}^1=0$ if and only if $\nu_{>0}^1=0$. The conclusion is then given by the equivalences (5) and (6). 
\end{proof}

\begin{corollary}
\label{gradedExt}
Let $A$ be a standard graded $k$-algebra. If $\Ext_A(k,k)$ is finitely generated over $A^{!}$, then  $\Ext_A(k,k)=A^{!}$ and $A$ is Koszul.
\end{corollary}

\begin{proof}
 Use  Proposition \ref{gradedmu1} and the graded version of the equivalence (6) in the Theorem to conclude that $\Ext_A(k,k)=A^!$. This is a known characterization of Koszul algebras.
\end{proof}

We say that $R$ is complete intersection if $\widehat R\cong Q/(f_1, \cdots, f_c)$ with $(Q,\fn,k)$ a regular local ring and ${\bd f}=f_1, \dots, f_c$ is a regular sequence. Note that $c=\codim R$.

\begin{corollary}
Assume $R$ is complete intersection.  If $\Ext_R(k,k)$ is finitely generated over $R^!$, then  $\Ext_R(k,k)=R^{!}$. 
\end{corollary}

\begin{proof}
Assume $R$ is complete intersection. In this case,  the equality $s(R)=1$ is equivalent to $\Ext_R(k,k)=R^{!}$; this can be seen from \cite[5.3, 5.4(2)]{S}. In particular, all implications in the diagram in Theorem \ref{implications} are reversible. 
\end{proof}

We now record some consequences in terms of the linearity defect, which follow immediately from the results above and Theorem \ref{mu}.

\begin{corollary}
If $\ld_R(k)<\infty$, then $s(R)=1$. 
\end{corollary}

\begin{corollary}
Assume that $\ld_R(k)<\infty$. If either $\ld_R(k)\le 2$ or $R$ is complete intersection, then $\Ext_R(k,k)=R^{!}$. 
\end{corollary}

\section{Complete intersection rings}

In this section we answer positively Question 1 for complete intersection rings, under the additional assumption that $\agr R$ is Cohen-Macaulay. 

\begin{chunk}
Set $\dim R=d$. If $M$ is a finitely generated $R$-module of dimension $d$, recall, cf. \cite[4.1.8, 4.1.9]{BH}, that
\begin{equation}
\label{Hilbfrac}
\hilb Mt=\frac{Q_M(t)}{(1-t)^d}
\end{equation}
for some polynomial $Q_M(t)\in \mathbb Z[t]$ with $Q_M(1)\ne 0$. We define the {\it multiplicity} of $M$, denoted  $e(M)$, by setting $e(M)=Q_M(1)$. We set $e(M)=0$ whenever $\dim M<d$. With this definition, $e(-)$ is exact on short exact sequences, see \cite[4.6.7]{BH}.
\end{chunk}

\begin{lemma}
\label{e}
Let $R$ be a $d$-dimensional Cohen-Macaulay local ring. Assume that  $\Po^R_M(t)=A(t)/B(t)$ with $A(t)$, $B(t)$ relatively prime  polynomials in $\mathbb Z[t]$. 

If $\ld_R(M)<\infty$, then $B(-1)\ne 0$ and $e(M)/e(R)=A(-1)/B(-1)$. 
\end{lemma}

\begin{proof}
Set $\beta_i=\beta_i^R(M)$. Let $K$ be a Koszul syzygy of $M$ so that $\depth K\ge \depth R$. Since $R$ is Cohen-Macaulay, we have that $\dim K=d$. 

We have an exact sequence: 
\begin{equation}
\label{syzygy}
0\to K\to R^{\beta_n}\to R^{\beta_{n-1}}\to\cdots\to R^{\beta_1}\to R^{\beta_0}\to M\to 0\,.
\end{equation}
Since $K$ is Koszul, \eqref{KoszulP} gives
\begin{equation}
\label{PoK}
\Po^R_K(t)=\frac{\hilb{ K}{-t}}{\hilb{R}{-t}}\,.
\end{equation}
As recalled in \eqref{Hilbfrac}, we also have 
\begin{equation}
\label{HilbKR}
\hilb{K}t=\frac{Q_K(t)}{(1-t)^d}\quad \text{and} \quad \hilb{R}t=\frac{Q_R(t)}{(1-t)^d} 
\end{equation}
for $Q_R(t), Q_K(t)\in \mathbb Z[t]$ with  $Q_K(1)=e(K)$ and $Q_R(1)=e(R)$. 

The hypothesis and the choice of $K$ give that 
\begin{equation}
\label{AB}
\frac{A(t)}{B(t)}=\Po^R_M(t)=\beta_0+\beta_1t+\cdots+\beta_nt^n+\Po^R_K(t)\cdot t^{n+1}\,.
\end{equation}

plugging in \eqref{PoK} and \eqref{HilbKR} into \eqref{AB}, we have: 
\begin{equation}
\label{AB'}
\frac{A(t)}{B(t)}=\beta_0+\beta_1t+\cdots+\beta_nt^n+\frac{Q_K(-t)}{Q_R(-t)}\cdot t^{n+1}\,.
\end{equation}

A multiplicity count in the exact sequence \eqref{syzygy} gives: 
\begin{equation}
\label{mcount}
(\beta_0-\beta_1+\cdots+(-1)^n\beta_n)e(R)=e(M)-(-1)^{n+1}e(K)
\end{equation}

We then plug in $t=-1$ into \eqref{AB'}. Note  that  $Q_K(1)=e(K)$ and $Q_R(1)=e(R)$. We then  use \eqref{mcount} in order to conclude:

\begin{align*}
\frac{A(-1)}{B(-1)}&=\frac{e(M)-(-1)^{n+1}e(K)}{e(R)}+\frac{Q_K(1)}{Q_R(1)}\cdot (-1)^{n+1}\\
&=\frac{e(M)-(-1)^{n+1}e(K)}{e(R)}+\frac{e(K)}{e(R)}\cdot (-1)^{n+1}=\frac{e(M)}{e(R)}
\end{align*}
\end{proof}

A sequence ${\bd x}=x_1, \dots, x_m$ is said to be {\it strictly regular} if the initial forms $x_1^*, \dots, x_m^*$ form a regular sequence in $\agr R$. A strictly regular sequence is, in particular, a regular sequence in $R$.

\begin{lemma}
\label{reduce}
If ${\bd x}\in\fm\smallsetminus \fm^2$ is a strictly regular sequence on $R$, then $\ld_R(k)=\ld_{R/({\bd x})}(k)$. 
\end{lemma}

\begin{proof}
We may assume that ${\bd x}$ consists of a single element $x$. Since $x$ is strictly regular, we have $(\fm^{n+1}\col x)\subseteq \fm^n$ for all $n$. Then apply \cite[8.7]{S} together with Theorem \ref{mu}.  
\end{proof}

\begin{remark}
\label{artinian1}
Let $R$ be a ring of dimension $d$ such that $\agr R$ is Cohen-Macaulay.  Consider the ring $R'=R[t]_{\fm[t]}$. The residue field of $R'$ is infinite and the natural map $R\to R'$ is faithfully flat. We can then choose a strictly regular sequence ${\bd g}=g_1, \dots, g_d$ such that the length of 
$R''=R'/({\bd g})$ is equal to $e(R)$. Note that $\edim R''=\codim R$.

Since the  map  of local rings $\varphi\colon (R,\fm,k)\to (R', \fm',k')$ is faithfully flat, with $\varphi(\fm)=\fm'$,  we have that $\ld_R(k)=\ld_{R'}(k)$. Furthermore, Lemma \ref{reduce} yields that $\ld_R(k)=\ld_{R''}(k)$. 
\end{remark}

If $R$ is complete intersection of codimension $c$,  then it is known $e(R)\ge 2^c$, cf.\! \cite[\S 7, Proposition 7]{B}.  We say that $R$ has {\it minimal multiplicity} if $e(R)=2^c$. 
\begin{theorem}
\label{ci}
If $R$ is a  complete intersection, then the following statements are equivalent: 
\begin{enumerate}[\quad\rm(a)] 
\item $\ld_R(k)=0$;
\item $R$ has minimal multiplicity. 
\end{enumerate}
Furthermore, if $\agr R$ is Cohen-Macaulay, then they are also equivalent to
\begin{enumerate}[\quad\rm(c)]
\item $\ld_R(k)<\infty$.
\end{enumerate}
\end{theorem}

\begin{remark}
If $R$ is a $d$-dimensional complete intersection ring of embedding dimension $e$ and codimension $c$, then $d=e-c$ and a result of Tate and Zariski, see for example  \cite[3.4.3]{GL},  gives 
\begin{equation}
\label{Pok}
\Po^R_k(t)=\frac{(1+t)^e}{(1-t^2)^c}=\frac{(1+t)^d}{(1-t)^c}\,.
\end{equation}
\end{remark}

\begin{proof}
${\rm (a)}\implies {\rm (b)}$: If $\ld_R(k)=0$, then $\agr R$ is a Koszul ring. We use then \eqref{KoszulP} and \eqref{Hilbfrac} to conclude 
$$
\Po^R_k(t)=\frac{1}{\hilb R{-t}}=\frac{(1+t)^d}{Q_R(-t)}\,.
$$
Comparing this  with the formula \eqref{Pok}, it  follows that $Q_R(t)=(1+t)^c$, hence $e(R)=Q_R(1)=2^c$. 

${\rm (b)}\implies{\rm (a)}$: If $R$ is a complete intersection of minimal multiplicity, then $\agr R$ is a complete interesection of quadrics, hence $\agr R$ is Koszul. 

Obviously, ${\rm (a)}\implies{\rm (c)}$. 

Assume now that $\agr R$ is Cohen-Macaulay. 

${\rm (c)}\implies{\rm (b)}$: Note that we may replace $R$ with the ring $R''$ of Remark \ref{artinian1}. We may assume thus that $d=0$.

Then \eqref{Pok} gives $\Po^R_k(t)=\frac{1}{(1-t)^c}$. Using Lemma \ref{e} with $M=k$,  we conclude $e(R)=2^c$, hence $R$ has minimal multiplicity. 
\end{proof}

\section{Artinian and Golod rings}

In this section we use a result of Martsinkovsky  to provide more evidence for Question 2  in the case of Artinian rings (Theorem \ref{artinian}). We also settle Question 1 in the case of Golod rings $R$ with $\agr R$ Cohen-Macaulay, using results of Avramov and Levin (Theorem \ref{golod}).  

\begin{theorem}
\label{artinian}
Assume $R$ is Artinian with $\fm^{n+1}=0$. If $\nu_{\gg 0}^{n-1}=0$, then $\nu_{>0}^{n-1}=0$.
\end{theorem}

\begin{proof}
For integers  $i$, $s$ denote 
$$\gamma_i^s\colon \Ext^{i}_R(\fm^{s-1},k)\to\Ext^{i}_R(\fm^{s},k)
$$
the map induced in cohomology by the inclusion $\fm^{s}\hookrightarrow\fm^{s-1}$. 

If $i>1$ and $s>0$, note that $\nu^s_i=0$ if and only if $\Hom_R(\nu^s_i,k)=0$. Using the natural isomorphisms $\Ext^i_R(R/\fm^{a},k)\cong \Ext^{i-1}_R(\fm^a,k)$ for $a=s,s+1$, we conclude that  $\nu^s_i=0$ if and only if $\gamma_{i-1}^{s+1}=0$.

Let $j>0$ and assume that $i$ large enough so that $\nu_{i+j+1}^{n-1}=0$, thus $\gamma_{i+j}^{n}=0$. 

Consider now the commutative diagram below, where where $\alpha=E^i\otimes \gamma^n_j$ and $\beta=\gamma^n_{i+j}=0$, the vertical arrows are given by the Yoneda product, and the isomorphisms are due to the fact that $\fm^n$ is a $k$-vector space, since $\fm^{n+1}=0$. The  rightmost square is induced by choosing a certain  projection $\Hom(\fm^n,k)\to k$; the choice of this projection will be discussed later. 
\[
\xymatrixrowsep{1.7pc}
\xymatrixcolsep{1.0pc}
\xymatrix{
E^i\otimes\Ext_R^j(\fm^{n-1},k)\ar@{->}[r]^{\alpha}\ar@{->}[d]&E^i\otimes \Ext_R^j(\fm^n,k)\ar@{->}[d]\ar@{->}[r]^{\qquad\qquad\cong\qquad\qquad\quad}& E^i\otimes E^j\otimes\Hom(\fm^n,k)\ar@{->}[d]\ar@{->}[r]& E^i\otimes E^j\ar@{->}[d]\\
\Ext_R^{i+j}(\fm^{n-1},k)\ar@{->}[r]^{\beta=0}&\Ext_R^{i+j}(\fm^n,k)\ar@{->}[r]^{\cong}&E^{i+j}\otimes\Hom(\fm^n,k) \ar@{->}[r]&E^{i+j}}
\]
Assume that $\nu_{j+1}^{n-1}\ne 0$, or, equivalently, that $\gamma^{n}_j\ne 0$. There exists thus an element in $\Ext_R^j(\fm^{n-1},k)$ whose image $\theta$ in $\Ext_R^j(\fm^n,k)$ under $\gamma^n_j$  is non-zero. The commutativity of the left  square yields that the Yoneda product $\varphi\theta$ is zero for all $\varphi\in E^i$. 

Let $\widetilde \theta$ denote the image of $\theta$ in $E^j\otimes\Hom(\fm^n,k)$. The commutativity of the middle square shows that 
$\varphi\widetilde \theta=0$ for all $\varphi\in E^i$. Since $\widetilde \theta\ne 0$, we can now construct the right square in the diagram by choosing a projection $\Hom(\fm^n,k)\to k$ in such a way that the image of $\widetilde \theta$ under the induced map $E^j\otimes \Hom(\fm^n,k)\to E^j\otimes k\xrightarrow{\cong} E^j$ remains nonzero. Let $\ov\theta$ denote this image. 

 The commutativity of the right square shows then that $\varphi\ov\theta=0$ for all $\varphi\in E^i$. The element $\ov\theta$ of $E$ is thus annihilated by all elements of $E$ of sufficiently large degree. This is a contradiction, according to the proof of \cite[Theorem 6]{M}.

The contradiction shows that $\nu_{j+1}^{n-1}=0$ for all $j>0$. Recalling that $\nu_1^{n-1}=0$ by \ref{mu_1}, we conclude that $\nu_{>0}^{n-1}=0$. 
\end{proof}

If $R$ is a Cohen-Macaulay ring, then one has an inequality $\codim R\le e(R)-1$. If equality holds, we say that $R$ has {\it minimal multiplicity} (as a Cohen-Macaulay ring). We talked earlier about minimal multiplicity for complete intersections, and one should distinguish between the two notions.  In particular,  note that a complete intersection ring $R$ has minimal multiplicity as a Cohen-Macaulay ring only when $\codim R\le 1$. 

Besides complete intersections, Golod rings constitute another major class of rings for which the homological behavior of modules is fairly well understood. Since the definition of such rings is somewhat technical, we refer to \cite[\S 5]{Avr98}  for the definition and properties. 

\begin{theorem}
\label{golod}
If $R$ is Golod and $\agr R$ is Cohen-Macaulay, then the following statements are equivalent: 
\begin{enumerate}[\quad\rm(a)]
\item $\ld_R(k)=0$;
\item $\ld_R(k)<\infty$;
\item $R$ has minimal multiplicity. 
\end{enumerate} 
\end{theorem}

\begin{remark}
\label{golodpreserve}
Assume that $R$ is as in the hypothesis of the Corollary, and let $R'$ and $R''$ as in Remark \ref{artinian1}. Since $\varphi\colon (R,\fm,k)\to (R', \fm',k')$ is faithfully flat, with $\varphi(\fm)=\fm'$, note that $R'$ is Golod and $\agr{R'}$ Cohen-Macaulay.  Also, since the strictly regular sequence ${\bd g}$ is contained in $\fm\smallsetminus \fm^2$, note that $R''$ is Golod, as well, see \cite[5.2.4]{Avr98}
\end{remark}

\begin{proof}
Using Remarks \ref{artinian1} and \ref{golodpreserve},  we may assume that $R$ is Artinian. 

(c)$\implies$(a): Since $R$ is Artinian of minimal multiplicity, we have $\fm^2=0$. One has $R\cong \agr R$ and $\agr R$ is obviously Koszul. 

(b)$\implies$(c): Let $n$ be such that $\fm^{n+1}=0$ and $\fm^n\ne 0$. By Theorem \ref{artinian} and Theorem \ref{mu}, we have that $\nu^{n-1}_{>0}=0$. If $n\ge 2$, then the natural map $R\to R/\fm^{n}$ is Golod (cf.\,Levin \cite[3.15]{Lev}) and thus small (cf.\,Avramov \cite[3.5]{Av}). Since the only small ideal of a Golod Artinian ring is $(0)$, see \cite[4.7]{Av}, we must have $\fm^n=0$, a contradiction.We conclude $n\le 1$, hence $\fm^2=0$.

Since (a) obviously implies (b), the proof is completed. 
\end{proof}


\begin{thebibliography}{99}



\bibitem{Avr98}
L.~L.~Avramov, \textit{Infinite free resolutions},  
Six lectures on commutative algebra (Bellaterra, 1996), 
Progress in Math. {\bf 166}, Birkh\"auser, Basel, 1998;  1--118.

\bibitem{Av}
L.~L.~Avramov, {\it Small homomorphisms of local rings}, J. Algebra {\bf 50} (1978), 400--453.

\bibitem{AIS}
L.~L.~Avramov, S.~B.~Iyengar, L.~M.~\c Sega, {\it Short Koszul modules}, J. Commut. Algebra {\bf 2} (2010), 249--279.

\bibitem{AP}
L.~L.~Avramov, I.~Peeva, \textit{Finite regularity and Koszul algebra}, Amer. J. Math. (2001), 275--281. 

\bibitem{B}
N.~Bourbaki, {\it Alg\`ebre commutative. Chapitre VIII: Dimension}, Masson, Paris, 1983. 

\bibitem{BH}
W.~Bruns, J.~Herzog, {\it Cohen-Macaulay rings}, Cambridge Studies in Advanced Mathematics {\bf 39},  Cambridge University Press, Cambridge, 1993. 

\bibitem{GL}
T.~H.~Gulliksen, G.~Levin, 
\textit{Homology of local rings}, 
Queen's Papers Pure Appl. Math. {\bf 20}, 
Queen's Univ., Kingston, ON (1969).



\bibitem{HI}
J.~Herzog, S.~Iyengar, \textit{Koszul modules}, J. Pure Appl. Algebra {\bf 201} (2005), 154--188.


\bibitem{IR} S.~Iyengar, T.~R\"omer, \textit{Linearity defects of modules over commutative rings}, J. Algebra {\bf 322} (2009), 3212--3237.

\bibitem{Lev} G.~Levin, \textit{Local rings and Golod homomorphisms}, J. Algebra {\bf 37} (1975), 266-289. 

\bibitem{M}
A.~Martsinkovsky, \textit{A remarkable property of the (co)syzygy modules of the residue field of a nonregular local ring},  J. Pure Appl. Algebra {\bf 110} (1996), 9--13. 

\bibitem{Ro}
J.-E.~Roos, {\it Relations between Poincar\'e-Betti series of loop spaces and of local rings}, S\'eminaire d'Alg\`ebre Paul Dubreil 31\`eme ann\'ee (Paris, 1977--1978), Lecture Notes in Math. {\bf 740}, Springer, Berlin, 1979; 285--322. 

\bibitem{S} L.~M.~\c Sega, \textit{Homological properties of powers of the maximal ideal of a local ring} {\bf 241} (2001), 827--858. 



\end{thebibliography}
\end{document}